\documentclass[12pt]{article}
\usepackage[margin=1in]{geometry}
\usepackage{amsmath,amssymb,amsthm}
\usepackage[hidelinks]{hyperref}

\newtheorem{theorem}{Theorem}
\newtheorem{lemma}[theorem]{Lemma}

\begin{document}

\title{The order of long rainbow arithmetic progressions}
\author{Jesse Geneson}
\date{}

\maketitle

\begin{abstract}
Let $T_k$ be the minimum positive integer $t$ such that, for every positive
integer $n$, every equinumerous $t$-coloring of $[tn]$ contains a rainbow
$k$-term arithmetic progression. Jungi\'{c}, Licht, Mahdian,
Ne\v{s}et\v{r}il and Radoi\v{c}i\'{c} conjectured that
$T_k=\Theta(k^2)$, while Conlon, Fox and Sudakov proved that
$T_k=O(k^2\log k)$. We prove the matching lower bound
$T_k=\Omega(k^2\log k)$, and hence $T_k=\Theta(k^2\log k)$.
\end{abstract}

\section{Introduction}

For each positive integer $m$, write $[m]=\{1,\ldots,m\}$, and use natural
logarithms unless a base is specified. A coloring is \emph{equinumerous} if
every color is used equally often. Throughout, an arithmetic progression has
positive common difference and is \emph{rainbow} if its terms have distinct
colors. For each integer $k\geq3$, let $T_k$ be the minimum positive integer
$t$ such that, for every positive integer $n$, every equinumerous $t$-coloring
of $[tn]$ contains a rainbow $k$-term arithmetic progression.

Jungi\'{c}, Licht, Mahdian, Ne\v{s}et\v{r}il and
Radoi\v{c}i\'{c}~\cite{JLMNR} introduced this problem. They proved that
$T_k=\Omega(k^2)$ and $T_k=O(k^3)$, and conjectured that
$T_k=\Theta(k^2)$. In an earlier note~\cite{G18}, we improved the upper bound to
$T_k\leq k^{5/2+o(1)}$. Balogh, Linz and Mattos~\cite{BLM} proved that
$T_k\leq k^2\exp\!\bigl((1+o(1))(\log\log k)^2\bigr)$, and Conlon, Fox and
Sudakov~\cite[Corollary~2.4]{CFS} proved the stronger bound
\[
                              T_k=O(k^2\log k).
\]

We prove that the logarithmic factor is necessary, refuting the conjecture
that $T_k=\Theta(k^2)$.

\begin{theorem}\label{thm:main}
There are positive constants $c$ and $C$ such that, for all sufficiently
large integers $k$,
\[
                      ck^2\log k\leq T_k\leq Ck^2\log k.
\]
Consequently, $T_k=\Theta(k^2\log k)$.
\end{theorem}

For every sufficiently large integer $k$ and every positive integer
$t\leq ck^2\log k$, we prove the lower bound by constructing an equinumerous
$t$-coloring of a suitable interval with no rainbow $k$-term arithmetic
progression.

The construction first partitions a slightly longer interval into four-point
blocking classes in residue lanes modulo a prime $M$, then restricts and
coarsens the partition to obtain $t$ equal classes. Differences divisible by
$M$ are handled within lanes. For a target interval of length $N$, the
remaining differences $d\leq P=O(N/k)$ with $M\nmid d$ are covered by a set
$\mathcal A$ of $O(P/\log k)$ integers such that each $d$ divides some
$D\in\mathcal A$ with $D/d\leq k-M$. The factor $\log k$ saved over the
trivial cover yields the lower bound.

\section{A multiplicative separation cover}\label{sec:cover}

For real $x\geq1$, put
\[
 \pi(x)=\#\{p\leq x:p\text{ is prime}\},
 \qquad
 \vartheta(x)=\sum_{\substack{p\leq x\\p\text{ prime}}}\log p,
 \qquad
 \psi(x)=\sum_{\substack{p\text{ prime},\ a\in\mathbb Z_{\geq1}\\p^a\leq x}}\log p.
\]

\begin{lemma}\label{lem:chebyshev-lcm}
There is an absolute constant $C_0\geq1$ such that, for every real $x\geq2$,
\begin{equation}\label{eq:chebyshev-lcm-bounds}
 \vartheta(x)\leq C_0x,\qquad \psi(x)\leq C_0x,\qquad
 \operatorname{lcm}(1,\ldots,\lfloor x\rfloor)\leq e^{C_0x},
 \qquad \pi(x)\leq C_0\frac{x}{\log x}.
\end{equation}
\end{lemma}

\begin{proof}
For every integer $u\geq1$, the product of the primes in $(u,2u]$ divides
$\binom{2u}{u}$. Hence
\[
 \vartheta(2u)-\vartheta(u)\leq\log\binom{2u}{u}<u\log 4.
\]
Telescoping over $u=1,2,4,\ldots,2^{m-1}$ gives
$\vartheta(2^m)<2^m\log 4$ for every integer $m\geq1$. Put
$K_\vartheta=2\log 4$. For every real $x\geq2$, taking the least integer
$m$ with $2^m\geq x$, so that $2^m<2x$, gives
\[
 \vartheta(x)\leq\vartheta(2^m)<2^m\log 4\leq K_\vartheta x.
\]
Only exponents
$a\leq\log_2x$ contribute to $\psi(x)$. Separating $a=1$, $a=2$, and
$a\geq3$ gives
\[
 \psi(x)=\sum_{a=1}^{\infty}\vartheta(x^{1/a})
 \leq K_\vartheta\bigl(x+\sqrt{x}+x^{1/3}\log_2x\bigr)
 \leq3K_\vartheta x,
\]
where the last inequality uses $\log_2x\leq x^{2/3}$ for $x\geq2$.
The logarithm of
$\operatorname{lcm}(1,\ldots,\lfloor x\rfloor)$ is $\psi(\lfloor x\rfloor)$.
For the prime-counting bound, splitting the primes at $\sqrt{x}$ gives
\[
 \pi(x)-\pi(\sqrt{x})
 \leq\frac{2\vartheta(x)}{\log x}
 \leq2K_\vartheta\frac{x}{\log x}.
\]
Since $\pi(\sqrt{x})\leq\sqrt{x}\leq x/\log x$ for $x\geq2$, we have
$\pi(x)\leq(1+2K_\vartheta)x/\log x$. Taking
$C_0=1+3K_\vartheta$ proves all four bounds.
\end{proof}

\begin{lemma}\label{lem:weighted-multisets}
Fix a real number $A\geq1$. There are constants $L_A\geq1$ and $C_A>0$,
depending only on $A$, such that the following holds. Let $L$ and $P$ be real
numbers with $L\geq L_A$ and $e^L\leq P\leq e^{2L}$. For every integer
$1\leq\ell\leq\lfloor\log_2P\rfloor$, let $\mathcal Q_\ell$ be a countable
collection of labels of weights $w(Q)\geq2$.
\begin{samepage}
If, for every such
$\ell$ and every real $x\geq2$,
\[
 \#\{Q\in\mathcal Q_\ell:w(Q)\leq x\}
 \leq A\left(1+\frac{\ell x}{L}\right),
\]
then
\[
 \sum_{1\leq\ell\leq\lfloor\log_2P\rfloor}
 \#\left\{\{Q_1,\ldots,Q_\ell\}_{\mathrm{multi}}:
                 Q_1,\ldots,Q_\ell\in\mathcal Q_\ell,\quad
                 \prod_{i=1}^{\ell}w(Q_i)\leq P\right\}
 \leq C_A\frac P L.
\]
\end{samepage}
Here $\{Q_1,\ldots,Q_\ell\}_{\mathrm{multi}}$ denotes a multiset.
\end{lemma}

\begin{proof}
Put $E=\lceil4A\rceil$. Choose $L_A\geq1$ so that
\begin{equation}\label{eq:exceptional-threshold}
 3L(3L+E)^{E-1}\leq\frac{e^L}{L}\qquad(L\geq L_A).
\end{equation}
Such a choice exists because $E$ depends only on $A$ and the left side is a
polynomial in $L$. Assume $L\geq L_A$.

We first record that, for every integer $r\geq1$ and every real $Y\geq1$,
\begin{equation}\label{eq:short-distinct-tuples}
 \#\{(n_1,\ldots,n_r)\in\mathbb Z_{\geq1}^r:
       n_1<\cdots<n_r,\ n_1\cdots n_r\leq Y\}
 \leq \frac{7^{r-1}Y(1+\log Y)^{r-1}}{r!(r-1)!}.
\end{equation}
Indeed, let $U_r(Y)$ count ordered $r$-tuples of positive integers of
product at most $Y$.  We claim that
\[
 U_r(Y)\leq \frac{Y(r+\log Y)^{r-1}}{(r-1)!}.
\]
The case $r=1$ is immediate, since $U_1(Y)=\lfloor Y\rfloor\leq Y$.  For
$r\geq2$, induction on $r$ gives
\[
 U_r(Y)=\sum_{\substack{n\in\mathbb Z_{\geq1}\\n\leq Y}}U_{r-1}(Y/n)
 \leq \frac{Y}{(r-2)!}
 \sum_{\substack{n\in\mathbb Z_{\geq1}\\n\leq Y}}
 \frac{(r-1+\log(Y/n))^{r-2}}n.
\]
The summand is a decreasing function of the real variable $n$ on $[1,Y]$,
so the sum is at most its value at $n=1$ plus the corresponding integral,
which the substitution $u=\log(Y/n)$ evaluates: with $B=r-1+\log Y$,
\begin{align*}
 \sum_{\substack{n\in\mathbb Z_{\geq1}\\n\leq Y}}
 \frac{(r-1+\log(Y/n))^{r-2}}n
 &\leq B^{r-2}+\int_0^{\log Y}(r-1+u)^{r-2}\,du\\
 &= B^{r-2}+\frac{B^{r-1}-(r-1)^{r-1}}{r-1}
 \leq \frac{(B+1)^{r-1}}{r-1},
\end{align*}
the last step because the binomial theorem gives
$(B+1)^{r-1}\geq B^{r-1}+(r-1)B^{r-2}$.  Since $B+1=r+\log Y$, the claim
follows.  Now suppose the set in \eqref{eq:short-distinct-tuples} is
nonempty.  Then $n_i\geq i$ for each $i$, so $Y\geq r!$ and hence
$\log Y\geq\log(r!)\geq\int_1^r\log x\,dx=r\log r-r+1$.  If $r\geq8$, then
$\log r\geq2$, so $\log Y\geq r$ and $r+\log Y\leq2\log Y$; if $r\leq7$, then
$r+\log Y\leq7(1+\log Y)$, since $\log Y\geq0$.  In both cases
$r+\log Y\leq7(1+\log Y)$, so
\[
 \#\{(n_1,\ldots,n_r)\in\mathbb Z_{\geq1}^r:
       n_1<\cdots<n_r,\ n_1\cdots n_r\leq Y\}
 \leq\frac{U_r(Y)}{r!}
 \leq\frac{Y\,7^{r-1}(1+\log Y)^{r-1}}{r!(r-1)!},
\]
after division by the $r!$ orderings of a tuple with distinct entries.
This proves \eqref{eq:short-distinct-tuples}.

Fix an integer $\ell$ in the range specified by the lemma and index the labels
of $\mathcal Q_\ell$ in nondecreasing order of weight as
$w_1\leq w_2\leq\cdots$, stopping if the list is finite. Such an indexing
exists because the counting hypothesis gives finitely many labels of weight
at most $x$ for every $x$. Call the first $E$ labels exceptional, or all labels
if fewer than $E$ exist. For every integer $m\geq1$ for which $w_{m+E}$
exists, the $m$th nonexceptional weight satisfies
\begin{equation}\label{eq:short-weight-rank}
                         w_{m+E}\geq\frac{mL}{2A\ell}.
\end{equation}
If $n$ is an integer with $n>E\geq2A$ and $w_n$ exists, then the first $n$
weights are at most $w_n$.
Thus the counting hypothesis applied at $x=w_n\geq2$ gives
$n\leq A(1+\ell w_n/L)$, so
$\ell w_n/L\geq n/A-1\geq n/(2A)$, using $n\geq2A$.  Hence
$w_n\geq nL/(2A\ell)$, and taking $n=m+E\geq m$ proves
\eqref{eq:short-weight-rank}.

Put
\[
 B_A=3+3\log(6A),\qquad K_A=14AB_A.
\]
We next claim that, for every integer $j\geq1$ and every real $X$ with
$0<X\leq P$, the number of $j$-element sets of nonexceptional labels of
$\mathcal Q_\ell$ with weight product at most $X$ is at most
\begin{equation}\label{eq:short-subset-bound}
 \frac XL\frac{(K_A\ell)^j}{j!(j-1)!}.
\end{equation}
If no such set exists, there is nothing to prove.  Otherwise, since every
weight is at least $2$, the existence of one such set forces
$2^j\leq X\leq P\leq e^{2L}$, so $j\leq2L/\log 2<3L$; also
$\ell\leq\log_2P\leq2L/\log 2<3L$ by the hypothesis of the lemma.  The
labels of any such $j$-set occupy $j$ distinct positions
$m_1<\cdots<m_j$ among the nonexceptional indices, and
\eqref{eq:short-weight-rank} gives
\[
 m_1\cdots m_j
 \leq\Bigl(\frac{2A\ell}{L}\Bigr)^{j}w_{m_1+E}\cdots w_{m_j+E}
 \leq Y,\qquad
 Y=\Bigl(\frac{2A\ell}{L}\Bigr)^{j}X.
\]
In particular $Y\geq1$.  Since $j<3L$, $\ell/L<3$ and $\log X\leq2L$,
\[
 1+\log Y\leq1+j\log(6A)+\log X\leq B_A L.
\]
The map sending a $j$-set of labels to its set of positions is injective,
so \eqref{eq:short-distinct-tuples} with $r=j$ bounds the number of such
$j$-sets by
\[
 \frac{7^{j-1}Y(1+\log Y)^{j-1}}{j!(j-1)!}
 \leq\frac{7^{j-1}(2A)^j\ell^jX(B_A L)^{j-1}}
 {L^j\,j!(j-1)!}
 \leq\frac XL\frac{(K_A\ell)^j}{j!(j-1)!},
\]
which proves \eqref{eq:short-subset-bound}. Since $\ell$ was arbitrary,
\eqref{eq:short-subset-bound} holds for every $\ell$ in the range specified by
the lemma.

For every integer $1\leq\ell\leq\lfloor\log_2P\rfloor$, call an $\ell$-element
multiset of labels of $\mathcal Q_\ell$ \emph{admissible} if its weight product
is at most $P$. It suffices to bound the total number of admissible multisets
over all $\ell$. Given such an $\ell$, let an admissible $\ell$-multiset have
$r$ distinct labels, of which $e\leq E$ are exceptional and $j=r-e$ are not,
and put $m=\ell-r\geq0$, so that $\ell=m+r=m+j+e$.

We first count the multisets for which $j\geq1$. For integers $m\geq0$,
$j\geq1$, and
$0\leq e\leq E$, put $r=j+e$ and $\ell=m+r$. If
$\ell\leq\lfloor\log_2P\rfloor$, let $N_{m,j,e}$ be the number of admissible
multisets with these parameters; otherwise put $N_{m,j,e}=0$. In the first
case, such a multiset is determined by its exceptional labels, its
nonexceptional labels, and its multiplicities. There are at most
$\binom{E}{e}$ choices for its exceptional labels, and a fixed set of $r$
labels carries exactly
$\binom{\ell-1}{r-1}=\binom{m+r-1}{r-1}$ positive multiplicity vectors.
The $m$ units of excess multiplicity contribute a factor at least $2^m$ to
the weight product. Thus the product of the distinct weights, and hence the
product of the nonexceptional weights, is at most $X=P2^{-m}$.
Applying \eqref{eq:short-subset-bound} gives
\begin{equation}\label{eq:short-repeat-term}
 \frac{N_{m,j,e}}{P/L}\leq
 \binom{E}{e}\binom{m+r-1}{r-1}2^{-m}
 \frac{(K_A(m+r))^j}{j!(j-1)!}.
\end{equation}
In the second case, \eqref{eq:short-repeat-term} is trivial. Thus we may sum
over all $m\geq0$, $j\geq1$, and $0\leq e\leq E$.

Set
\[
 z=2^{-1/2},\qquad \alpha=(1-z)^{-1},\qquad
 D_A=2\alpha^2+2\alpha\exp(E+1),\qquad
 G_A=\alpha K_A D_A.
\]
Since $\binom{m+r-1}{r-1}$ is the coefficient of $u^m$ in $(1-u)^{-r}$
and all coefficients are nonnegative, evaluating the series at $u=z$ and
using $z^2=1/2$ gives
\[
 \binom{m+r-1}{r-1}2^{-m}
 \leq\alpha^r z^m.
\]
For $j\geq1$ and $r=j+e\leq j+E$, convexity gives
\begin{align*}
 \sum_{m\geq0}z^m(m+r)^j
 &\leq2^{j-1}\left(
      \sum_{m\geq0}z^m m^j+\frac{r^j}{1-z}\right)\\
 &\leq2^{j-1}\left(
      \frac{j!}{(1-z)^{j+1}}+\frac{(j+E)^j}{1-z}\right)
 \leq D_A^j j!.
\end{align*}
For the second inequality, we used
$m^j\leq j!\binom{m+j}{j}$ and
$\sum_{m\geq0}\binom{m+j}{j}z^m=(1-z)^{-j-1}$. For the last inequality,
we used
\[
 (j+E)^j\leq\exp(E)j^j\leq\exp(E+j)j!,
\]
where the second bound follows from
$\log(j!)\geq\int_1^j\log x\,dx=j\log j-j+1$. After multiplication by
$2^{j-1}$, the terms $j!/(1-z)^{j+1}$ and $(j+E)^j/(1-z)$ are at most
$(2\alpha^2)^j j!$ and $(2\alpha\exp(E+1))^j j!$, respectively. Since $a^j+b^j\leq(a+b)^j$ for
$a,b\geq0$ and $j\geq1$, their sum is at most $D_A^j j!$.

It follows from \eqref{eq:short-repeat-term} that, for fixed $j$ and $e$,
\[
 \frac{1}{P/L}\sum_{m\geq0}N_{m,j,e}
 \leq\binom{E}{e}\alpha^e\frac{G_A^j}{(j-1)!}.
\]
Therefore the number of admissible multisets with $j\geq1$ is at most
\begin{align*}
 \frac PL
 \sum_{e=0}^E\binom{E}{e}\alpha^e
 \sum_{j\geq1}\frac{G_A^j}{(j-1)!}
 &=\frac PL(1+\alpha)^E G_A\exp(G_A).
\end{align*}
Put $F_A=(1+\alpha)^E G_A\exp(G_A)$.

If $j=0$, every label of the multiset is exceptional, so the multiset is
an $\ell$-multiset drawn from a set of at most $E$ labels, and there are
at most $\binom{\ell+E-1}{E-1}\leq(\ell+E)^{E-1}$ of these. Since
$\lfloor\log_2P\rfloor<3L$, their total number is at most
$3L(3L+E)^{E-1}\leq e^L/L\leq P/L$ by
\eqref{eq:exceptional-threshold}. Taking $C_A=F_A+1$ proves the lemma.
\end{proof}

\begin{lemma}\label{lem:separation-cover}
\begin{samepage}
There are absolute constants $J_0\geq2$ and $C_1>0$ such that the following
holds. Let $J$ and $P$ be real numbers with $J\geq J_0$ and
$J\leq P\leq J^2$, and let $M$ be prime. Then there is a set $\mathcal A$
of positive integers, none divisible by $M$, with
\[
                         |\mathcal A|\leq C_1\frac P{\log J},
\]
such that every positive integer $d\leq P$ with $M\nmid d$ has some
$D\in\mathcal A$ satisfying
\[
                         d\mid D,\qquad D/d\leq J.
\]
\end{samepage}
\end{lemma}

\begin{proof}
Let $C_0$ be the constant in Lemma~\ref{lem:chebyshev-lcm}, and put
$\Gamma=2+4C_0$. Let $L_{\Gamma}$ and $C_{\Gamma}$ be constants supplied by
Lemma~\ref{lem:weighted-multisets} for $A=\Gamma$. Choose $J_0\geq2$ so that
\[
 \log J\geq L_{\Gamma},\qquad \frac{J}{\log J}\geq1
 \qquad(J\geq J_0),
\]
and put $C_1=C_{\Gamma}+1$.

Fix $J\geq J_0$, $J\leq P\leq J^2$, and a prime $M$. For every integer
$1\leq\ell\leq\lfloor\log_2P\rfloor$, put $R_\ell=J^{1/\ell}$. List all
primes except $M$ in increasing order. Starting with the least unused prime
$p_0$, append successive unused primes as long as the product of the appended
primes remains at most $R_\ell$. When the next prime would make this product
exceed $R_\ell$, leave it unused and close the block. Since the product would
eventually exceed $R_\ell$ if the process continued, each block is finite;
repeating the procedure partitions the prime list. If $Q$ is a block and
$\Pi_Q=\prod_{p\in Q}p$, then
\begin{equation}\label{eq:short-block-quotient}
                         \Pi_Q/p\leq R_\ell\qquad(p\in Q).
\end{equation}
For $p=p_0$, this is the defining property of the block. Replacing $p_0$ by
a larger prime of the block decreases the quotient.

Let $B_\ell(x)$ count blocks whose least prime is at most $x$, and put
$\rho=\log R_\ell=(\log J)/\ell$.  Since $P\leq J^2$ and
$\ell\leq\lfloor\log_2P\rfloor$, we have $\rho\geq(\log 2)/2$. We claim that
\begin{equation}\label{eq:short-block-count}
 B_\ell(x)\leq\Gamma\left(1+\frac{x}{\rho}\right)\qquad(x\geq2).
\end{equation}
If $\rho\leq2$, this follows from
$B_\ell(x)\leq\pi(x)\leq x\leq2x/\rho$. Otherwise put
$y=\min\{x,\exp(\rho/2)\}$. The blocks with least prime at most $y$ form an
initial segment of the block sequence; apart from the last of them, every
such block has a successor beginning with a prime $q\leq y$. If $E_Q$ is
the product of its appended primes, greedy maximality gives
$E_Qq>\exp(\rho)$, and hence, since $q\leq\exp(\rho/2)$,
$\log E_Q>\rho/2$. The appended-prime sets are disjoint and contained in
$[2,x]$, so the number of these initial blocks is at most
$1+2\vartheta(x)/\rho\leq1+2C_0x/\rho$. If $x>\exp(\rho/2)$, the remaining
block minima lie in $(\exp(\rho/2),x]$ and number at most
\[
 \pi(x)\leq C_0\frac{x}{\log x}<2C_0\frac{x}{\rho}.
\]
Thus $B_\ell(x)\leq1+4C_0x/\rho$, which implies
\eqref{eq:short-block-count} by the definition of $\Gamma$.

Put $D(1)=1$. For every integer $d$ with $2\leq d\leq P$ and $M\nmid d$,
write its prime factorization as $d=p_1\cdots p_\ell$, with multiplicity.
Then $\ell\leq\lfloor\log_2P\rfloor$, so the partition corresponding to
$\ell$ is defined.
\begin{samepage}
Using that partition, define
\[
                         D(d)=\prod_{i=1}^{\ell}\Pi_{Q(p_i)},
\]
where $Q(p)$ denotes the block containing $p$.
\end{samepage}
\begin{samepage}
Equation \eqref{eq:short-block-quotient} gives
\[
                         d\mid D(d),\qquad D(d)/d
 =\prod_{i=1}^{\ell}\frac{\Pi_{Q(p_i)}}{p_i}\leq R_\ell^\ell=J,
\]
and $M\nmid D(d)$ because $M$ belongs to no block.
\end{samepage}

Set
\[
 \mathcal A=\{D(d):d\in\mathbb Z_{>0},\ d\leq P,\ M\nmid d\}.
\]
Give each block the
weight of its least prime; all weights are at least $2$.  The value $D(d)$
depends only on the multiset of blocks $Q(p_1),\ldots,Q(p_\ell)$, whose
weight product is at most $p_1\cdots p_\ell=d\leq P$. For every integer
$1\leq\ell\leq\lfloor\log_2P\rfloor$, \eqref{eq:short-block-count} verifies
the hypothesis of
Lemma~\ref{lem:weighted-multisets} with $A=\Gamma$ and $L=\log J$, since
\[
 \Gamma(1+x/\rho)=\Gamma(1+\ell x/\log J).
\]
Moreover, $e^L=J\leq P\leq J^2=e^{2L}$ and $L=\log J\geq L_{\Gamma}$. The
lemma therefore bounds the number of multisets of blocks with weight product
at most $P$, and hence the number of distinct values $D(d)$ with
$2\leq d\leq P$ and $M\nmid d$, by $C_{\Gamma}P/\log J$. Finally,
$P/\log J\geq J/\log J\geq1$, so
\[
 |\mathcal A|\leq C_{\Gamma}\frac P{\log J}+1
 \leq C_1\frac P{\log J}.
\]
\end{proof}

\section{The coloring construction}

We work on intervals of the form $\{0,\ldots,N-1\}$; translation by one
gives the interval $[N]$.

\begin{lemma}\label{lem:lane-coloring}
Let $k\geq3$ be an integer, let $M$ be prime, let $H,L$ be positive
integers, let $N_0=ML$, and put
\[
 P=\left\lfloor\frac{N_0-1}{k-1}\right\rfloor.
\]
\begin{samepage}
Suppose that
\begin{enumerate}
\item $4H\mid L$ and $2H\leq k-1$;
\item every positive integer at most $\lfloor P/M\rfloor$ divides $H$;
\item there is a set $\mathcal A$ of fewer than $M/4$ positive integers,
none divisible by $M$, such that, for every positive integer $d\leq P$ with
$M\nmid d$, some $D\in\mathcal A$ satisfies
\[
 D=jd\qquad\text{for some integer }1\leq j\leq k-M.
\]
\end{enumerate}
\end{samepage}
Then $\{0,\ldots,N_0-1\}$ has a partition into four-point classes such that
every $k$-term arithmetic progression in the interval contains two distinct
terms in one class.
\end{lemma}

\begin{proof}
We first record the internal matching used in every residue class. On
$\mathbb{Z}_L$, join
\[
 u+2vH\quad\text{to}\quad u+(2v+1)H
\]
for integers $u$ and $v$ with $0\leq u<H$ and $0\leq v<L/(2H)$. This is a
perfect matching on every cycle generated by addition of $H$, and every
matching edge joins two elements of $\mathbb{Z}_L$ that differ by $H$. If
$e$ is a positive divisor of $H$ and $h=H/e$, then the terms at indices
$0,h,2h$ of a $k$-term integer progression of difference $e$ contained in
$\{0,\ldots,L-1\}$ are $y,y+H,y+2H$, where $y$ is its first term. These
indices exist because $2h\leq2H\leq k-1$, and the matching edge covering
$y+H$ joins it to one of its two $H$-neighbors $y$ and $y+2H$ in
$\mathbb{Z}_L$, so one of the two adjacent pairs is matched. The same holds
for every cyclic translate of the matching, whose edges again join elements
differing by $H$.

Write each point uniquely as $r+Mq$, with integers $r$ and $q$ satisfying
$0\leq r<M$ and $0\leq q<L$; call the set with fixed $r$ a lane. For every
$D\in\mathcal A$, choose disjoint lane pairs
\[
 \{r_D,r_D+D\}\subset\mathbb{Z}_M.
\]
This can be done greedily. After $a$ pairs have been chosen, their $2a$
used residues forbid at most $4a$ choices of $r_D$, namely those with
$r_D$ or $r_D+D$ already used. Since $a<|\mathcal A|<M/4$, we have
$4a<M$, so a choice remains. Also $r_D\neq r_D+D$ in $\mathbb{Z}_M$
because $M\nmid D$.

Put the internal matching on lane $r_D$. For $x$ in this lane, define
$F_D(x)$ to be the unique element of $\{0,\ldots,N_0-1\}$ satisfying
\[
 F_D(x)\equiv x+D\pmod{N_0}.
\]
Then $F_D$ maps lane $r_D$ bijectively onto lane $r_D+D$ and induces a cyclic
translation on their $q$-coordinates. For each matching edge $\{x,y\}$ in lane $r_D$,
make
\[
 \{x,y,F_D(x),F_D(y)\}
\]
one class. Its four points are distinct because the two lanes are disjoint
and $F_D$ is injective. These classes partition the two lanes, and the
classes on lane $r_D+D$ pair its points according to a cyclic translate of
the internal matching. On every unused lane put
the internal matching and group its edges in pairs. This is possible because
$4\mid L$. We have partitioned all $N_0$ points into four-point classes.
By construction, every matched pair in every lane lies in one class, whether
the lane carries the internal matching or a cyclic translate.

Let $a,a+d,\ldots,a+(k-1)d$ be a $k$-term arithmetic progression in
$\{0,\ldots,N_0-1\}$, where $d\geq1$. Since
$a+(k-1)d\leq N_0-1$, we have $d\leq P$.

If $M\mid d$, write $d=Me$, where $e$ is a positive integer. Then all terms
lie in one lane, and their $q$-coordinates form a $k$-term progression of
difference $e\leq\lfloor P/M\rfloor$. Hence $e\mid H$.
The matching observation above gives a matched pair of distinct terms in this
lane, and that pair lies in one class.

Suppose instead that $M\nmid d$. Choose $D=jd$ as in the third hypothesis,
with $j$ an integer and $1\leq j\leq k-M$.
There are $k-j\geq M$ eligible lower indices $0\leq i\leq k-1-j$.
Because $M$ is prime and $M\nmid d$, the lane residues of any $M$
consecutive terms are all of $\mathbb{Z}_M$. Thus some eligible $i$ has
$a+id$ in lane $r_D$.
\begin{samepage}
Both selected points are actual progression terms, so
\[
 F_D(a+id)=a+(i+j)d
\]
without a cyclic wrap.
\end{samepage}
The terms $a+id$ and $a+(i+j)d$ belong to the same
four-point class. Thus in either case two distinct terms lie in one class.
\end{proof}

\begin{lemma}\label{lem:prefix-coarsening}
Let $N_0$ be a positive integer, let $\mathcal C$ partition
$\{0,\ldots,N_0-1\}$ into four-point sets, let $N$ be a positive multiple of
$12$, and put $R=N_0-N\geq0$. If $N>12R$,
then $\{0,\ldots,N-1\}$ can be partitioned into $N/12$ twelve-point sets,
each of which is a union of nonempty intersections
\[
 C\cap\{0,\ldots,N-1\},\qquad C\in\mathcal C.
\]
\end{lemma}

\begin{proof}
Call the nonempty intersections fragments. Let $F$ be the number of
four-point fragments, and let $b$ and $W$ be the number and total size of
the remaining fragments. Each of the latter comes from a distinct class of
$\mathcal C$ meeting the complement of the prefix, so $b\leq R$. At most
$R$ classes of $\mathcal C$ meet that complement, and hence
\[
 F\geq N_0/4-R\geq N/4-R.
\]
Since $N=4F+W$ and $12\mid N$, we have $4\mid W$.

Partition the smaller fragments into groups of total size $4$, $8$, or
$12$. To do this, whenever at least four fragments remain, order any four
of their sizes and consider the five partial sums, beginning with zero,
modulo $4$. Two partial sums are congruent modulo $4$, so a nonempty
consecutive subcollection of at most four fragments has total size divisible
by $4$, and its total is at most $12$. Remove it and repeat. Fewer than four
fragments eventually remain;
their total is still divisible by $4$ and, if nonzero, is $4$ or $8$.

Let the resulting group sizes be $w_1,\ldots,w_g$. Then $g\leq b\leq R$.
Pad a group of size $w_h$ to size $12$ using $(12-w_h)/4$ four-point
fragments. At most $2g\leq2R$ full fragments are required, while
\[
 F\geq N/4-R>2R.
\]
Thus the padding is possible. The number of full fragments used is
\[
 a_0=\sum_{h=1}^g\frac{12-w_h}{4}=3g-W/4.
\]
The number of full fragments left is therefore
\[
 F-a_0=F+W/4-3g=N/4-3g=3(N/12-g).
\]
Group these remaining fragments three at a time. Together with the $g$
padded groups, this gives exactly $N/12$ sets of size $12$ without
splitting a fragment.
\end{proof}

\section{Proof of the lower bound}

\begin{theorem}\label{thm:lower-bound}
There is an absolute constant $c>0$ such that, for every sufficiently large
integer $k$,
\[
 T_k>c k^2\log k.
\]
\end{theorem}

\begin{proof}
Let $C_0$ and $C_1$ be the constants in
Lemmas~\ref{lem:chebyshev-lcm} and \ref{lem:separation-cover}. Fix a
positive constant $c$ so small that
\[
 61C_0c<\frac14,
 \qquad 12C_1c<\frac1{40}.
\]
Apply Bertrand's postulate with $x=\lfloor k/4\rfloor$. For all sufficiently
large $k$, it gives a prime $M$ with
\[
 \frac{k}{5}<M<\frac{k}{2}.
\]
Put $J=k-M$, so $k/2<J<4k/5$.

Fix any integer $t$ with
\[
 1\leq t\leq\lfloor c k^2\log k\rfloor.
\]
All asymptotic notation in this proof is as $k\to\infty$ and is uniform over
this range of $t$.
Define
\[
 \begin{gathered}
 s=\left\lceil\frac{k^2}{t}\right\rceil,
 \qquad N=12st,
 \qquad \lambda=N/M,
 \qquad q=\left\lceil\frac{\lambda}{k-1}\right\rceil,\\
 H=\operatorname{lcm}(1,\ldots,q),
 \qquad L=4H\left\lceil\frac{\lambda}{4H}\right\rceil,\\
 N_0=ML,
 \qquad R=N_0-N,
 \qquad P=\left\lfloor\frac{N_0-1}{k-1}\right\rfloor.
 \end{gathered}
\]
We have
\[
 k^2\leq st<k^2+t,
 \qquad 0\leq R=M(L-\lambda)<4MH,
\]
the last inequality because $L<\lambda+4H$.
Since $N<12(k^2+t)\leq12k^2(1+c\log k)$ and $M>k/5$,
for all sufficiently large $k$,
\[
 q\leq\frac{N}{M(k-1)}+1
 <\frac{60k}{k-1}(1+c\log k)+1
 \leq62+61c\log k,
\]
using $60k/(k-1)\leq61$ for $k\geq61$.
Moreover, $N\geq12k^2$ and $M<k/2$, so
\[
 q\geq\frac{\lambda}{k-1}=\frac{N}{M(k-1)}
 >\frac{24k}{k-1}>2.
\]
The least-common-multiple estimate \eqref{eq:chebyshev-lcm-bounds}
therefore gives
\[
 H\leq e^{C_0q}\leq e^{62C_0}k^{61C_0c}.
\]
Since $61C_0c<\frac14$, this makes $H=o(k^{1/4})$, so, uniformly in $t$,
\[
 4H\leq k-1,
 \qquad R<4MH<2ke^{62C_0}k^{61C_0c}=o(k^{5/4}),
 \qquad N\geq12k^2>12R
\]
for all sufficiently large $k$.

Put $Q=\lfloor P/M\rfloor$. Since $P<ML/(k-1)$ and
$L<\lambda+4H\leq\lambda+k-1$, we have
\[
 Q<\frac{L}{k-1}
   <\frac{\lambda}{k-1}+1
   \leq q+1.
\]
Thus $Q\leq q$, and every positive integer at most $Q$ divides $H$.

Since $N_0=N+R<12(k^2+t)+o(k^2)$, the bounds on $s$, $t$, and $R$ give
\begin{align*}
 P&\geq\frac{12k^2-1}{k-1}-1=(12+o(1))k>J,\\
 P&<\frac{N_0}{k-1}\leq(12c+o(1))k\log k<J^2.
\end{align*}
Since $J>k/2\to\infty$, Lemma~\ref{lem:separation-cover} applies. It gives a
set $\mathcal A$, none of whose members is divisible by $M$, such that every
positive integer $d\leq P$ with $M\nmid d$ divides some $D\in\mathcal A$ with
$D/d\leq J=k-M$. Moreover, since
$\log J\geq\log(k/2)=(1-o(1))\log k$ and $M/4>k/20$,
\[
 |\mathcal A|
 \leq C_1\frac{P}{\log J}
 \leq(12C_1c+o(1))k
 <\frac M4,
\]
the last inequality by $12C_1c<1/40$. The relations $4H\mid L$,
$2H\leq k-1$, and $Q\leq q$ verify the first two hypotheses of
Lemma~\ref{lem:lane-coloring}. The cover and its size bound verify the third:
for each positive integer $d\leq P$ with $M\nmid d$, choose $D$ as above and
put $j=D/d$. Then $j$ is an integer with $1\leq j\leq J=k-M$. Thus
$\{0,\ldots,N_0-1\}$ has a partition into four-point classes such that every
$k$-term arithmetic progression in the interval contains two distinct terms
in one class.

Apply Lemma~\ref{lem:prefix-coarsening} to the prefix of length $N=12st$.
It gives $st$ twelve-point sets, each a union of intersections with the old
classes. Group these sets in batches of $s$. This partitions the prefix into
$t$ classes of size $12s$. The two blocked terms of any progression lie in
the same old-class fragment, and no fragment is split by the coarsening.
Thus the resulting coloring has no rainbow $k$-term arithmetic progression.

Since $t$ was arbitrary, for every integer $t$ with
\[
 1\leq t\leq\lfloor c k^2\log k\rfloor,
\]
translation by one gives an equinumerous $t$-coloring of $[12st]$ with no
rainbow $k$-term arithmetic progression. Since $12st=t(12s)$, this is a
counterexample to the defining property of $T_k$ with the positive integer
$n=12s$. Hence no such $t$ satisfies that property, so
$T_k>\lfloor c k^2\log k\rfloor$. Since $T_k$ is an integer,
$T_k>c k^2\log k$.
\end{proof}

\section*{Acknowledgments}

Codex with GPT-5.6 and Claude Code with Fable 5 were used
for proof exploration, proof criticism, exposition, and revision.


\begin{thebibliography}{9}
\raggedright

\bibitem{JLMNR}
V. Jungi\'{c}, J. Licht, M. Mahdian, J. Ne\v{s}et\v{r}il and
R. Radoi\v{c}i\'{c},
\emph{Rainbow arithmetic progressions and anti-Ramsey results},
\emph{Combinatorics, Probability and Computing} \textbf{12} (2003),
599--620.

\bibitem{G18}
J. Geneson,
\emph{A note on long rainbow arithmetic progressions},
arXiv:1811.07989, 2018.

\bibitem{BLM}
J. Balogh, W. Linz and L. Mattos,
\emph{Long rainbow arithmetic progressions},
\emph{Journal of Combinatorics} \textbf{12} (2021), no.~3, 547--550.

\bibitem{CFS}
D. Conlon, J. Fox and B. Sudakov,
\emph{Short proofs of some extremal results III},
\emph{Random Structures \& Algorithms} \textbf{57} (2020), no.~4,
958--982.

\end{thebibliography}
\end{document}